\pdfoutput=1
\documentclass[11pt,reqno]{amsart}
\usepackage[T1]{fontenc}
\usepackage{amsmath,amssymb,amsthm,amscd,lmodern}
\usepackage[margin=1in]{geometry}
\usepackage{microtype}
\usepackage{needspace}
\usepackage[round,authoryear]{natbib}
\usepackage{hyperref}
\hypersetup{hidelinks,
 pdftitle={A natural Chow lemma for smooth threefolds and fourfolds},
 pdfauthor={Parsa Bakhtary}}
\theoremstyle{plain}
\newtheorem{theorem}{Theorem}[section]
\newtheorem{proposition}[theorem]{Proposition}
\newtheorem{lemma}[theorem]{Lemma}

\theoremstyle{definition}
\newtheorem{definition}[theorem]{Definition}
\theoremstyle{remark}
\newtheorem{remark}[theorem]{Remark}
\DeclareMathOperator{\Pic}{Pic}
\DeclareMathOperator{\Bl}{Bl}
\DeclareMathOperator{\Gal}{Gal}
\DeclareMathOperator{\Spec}{Spec}
\DeclareMathOperator{\Supp}{Supp}
\DeclareMathOperator{\length}{length}
\DeclareMathOperator{\Nef}{Nef}
\newcommand{\cO}{\mathcal O}
\newcommand{\cI}{\mathcal I}
\newcommand{\cJ}{\mathcal J}
\newcommand{\cK}{\mathcal K}
\newcommand{\cM}{\mathcal M}
\newcommand{\cG}{\mathcal G}
\newcommand{\PP}{\mathbb P}
\newcommand{\ZZ}{\mathbb Z}
\newcommand{\RR}{\mathbb R}
\title[A natural Chow lemma]
{A natural Chow lemma for\\
smooth threefolds and fourfolds}
\author[P.~Bakhtary]{Parsa Bakhtary}
\address{Independent, Sunnyvale, California, USA}
\email{pbakhtary@gmail.com}
\urladdr{https://orcid.org/0000-0002-4491-5216}
\date{}

\begin{document}
\begin{abstract}
Let $X$ be a smooth proper integral variety of dimension at most four
over a field of characteristic zero. We construct a finite sequence of
blowups along smooth centers with projective endpoint, preserving the
intersection of all maximal quasi-projective opens of $X$.
The construction depends only on a fixed functorial principalization
procedure and is natural under isomorphisms, including those over
different ground fields. The proof has two parts. A colength argument
on regular local surfaces gives,
in every dimension, a finite preparation after which some maximal
quasi-projective open has complement mapping into codimension at least
three in $X$. In dimensions at most four these images are points or
curves. Ample numerical classes on the normalized components of that
complement give finite bounds for summing base ideals of linear systems.
The resulting ideal descends to the original field; geometric
integrality is not required.
\end{abstract}
\maketitle

\section{Introduction}

Chow's lemma and principalization give a projective smooth modification
of a smooth proper variety by blowups along smooth centers. They also
allow one to preserve a prescribed quasi-projective open. Can the
sequence and a preserved open be specified from the variety itself,
compatibly with isomorphisms?

\citet[\S3.3]{Rydh} points out that the choices of a quasi-projective
open and its projective compactification prevent his Chow construction
from being functorial, and leaves open how far functoriality is possible.
We obtain naturality under isomorphisms in dimension at most four.
This is weaker than compatibility with the smooth base changes appearing
in Rydh's other results. The open we preserve is the intersection of
all maximal quasi-projective opens. Nonprojective smooth proper
threefolds already occur in Hironaka's construction
\citep[Tag~08J0]{Stacks}.

\subsection{Conventions and the main result}
Throughout, $k$ is a field of characteristic zero. A $k$-variety is an
integral separated scheme of finite type over $k$; it need not be
geometrically integral. Projectivity and quasi-projectivity are over
the indicated ground field. Centers may be disconnected. We write
$\Bl_{\cI}Y$ for the blowup of a coherent ideal $\cI$ on $Y$, formed
from its Rees algebra. The \emph{unit locus} of $\cI$ is the open where
$\cI=\cO_Y$; its \emph{invertibility locus} is the open where it is an
invertible sheaf. The \emph{boundary} of an open $W\subset Y$ is the
reduced complement, and its components are the irreducible components
of that complement.

For a smooth proper $k$-variety $X$, let $\cM(X)$ be the set of its
maximal opens that are quasi-projective over $k$. Its finiteness over
an algebraically closed field was established by
\citet[Theorem~A]{Wlodarczyk1999}; for the arbitrary-field statement we
use \citet[Theorem~9]{Benoist}. It is nonempty by Noetherianness. Put
\[
 U(X)=\bigcap_{W\in\cM(X)}W.
\]
We call this dense quasi-projective open the \emph{quasi-projective
core}. Every isomorphism of $X$ preserves it. It is a common open
contained in all maximal quasi-projective opens; we make no maximality
claim for it among opens that a natural modification might preserve.

An isomorphism between varieties over possibly different fields is
\emph{semilinear} if it lies over an isomorphism of their ground fields.
A rule assigning ordered blowup sequences is \emph{natural under
isomorphisms} if every such isomorphism of inputs lifts to isomorphisms
of the ordered stages, commuting with the projections and carrying
each center onto its counterpart. All our blowups are isomorphisms
on a dense open. The lifts are therefore unique, and respect
identities and composition.

Fix the characteristic-zero principalization procedure of
\citet[Theorems~1.0.1 and 2.4.2, \S5.3]{Wlodarczyk}.
Principalization makes the total transform of a nonzero ideal
invertible by a finite sequence of smooth blowups. We use the
procedure's compatibility with semilinear isomorphisms and the fact
that its centers avoid the unit locus. Throughout, we omit empty
centers and Cartier components of smooth centers; these contribute
only identity blowups.

\Needspace{16\baselineskip}
\begin{theorem}\label{thm:main}
There is a rule, depending only on the fixed principalization
procedure, assigning to every characteristic-zero field $k$ and every
smooth proper $k$-variety $X$ of dimension at most four a finite
ordered sequence
\[
 X=X_0 \longleftarrow X_1\longleftarrow\cdots\longleftarrow X_r
\]
with the following properties:
\begin{enumerate}
\item Each arrow is the blowup of a smooth center of codimension at
least two, and $X_r$ is smooth and projective over $k$.
\item Every composite $X_i\to X$ is an isomorphism over $U(X)$.
\item The rule is natural under isomorphisms, including semilinear ones.
\item If $X$ is projective, the sequence is empty.
\end{enumerate}
\end{theorem}

In particular, every $k$-automorphism of $X$ lifts uniquely to every
stage. We assert naturality for individual isomorphisms, with no claim
about families or group-scheme actions. The rule is non-effective:
it uses least integers defined by existence and ampleness conditions.
Compatibility with recomputing the rule after a field extension is
not asserted.

\begingroup\emergencystretch=1em
Smooth-morphism functoriality would require additional compatibility:
neither the quasi-projective core nor the least-integer selections used
below come with a pullback rule, pulled-back admissible systems need not
remain admissible, and positive relative dimension can turn curve
boundary images into surfaces. Once principalization is fixed, the
construction is canonical in the precise sense that every isomorphism
determines a unique compatible isomorphism of the entire ordered tower.
\par\endgroup

\subsection{Background and known results}
For the open-preserving Chow lemma we use \citet[Theorem~3.5]{Rydh};
his Theorem~3.1(ii) refines a supplied modification to a sequence
with smooth centers. These results do not themselves specify the
initial projective modification naturally from $X$.

For comparison over $\mathbb C$, Fujiki's class $\mathcal C$ consists
of reduced compact complex spaces bimeromorphic to compact K\"ahler
manifolds. \citet[Theorem~1.1]{Kim} states a K\"ahler-model theorem
equivariant under the full automorphism group for spaces in this class.
A smooth proper complex algebraic variety is Moishezon and hence
belongs to this class. Applied to it, that conclusion would give
a K\"ahler manifold bimeromorphic to a projective variety,
hence a projective manifold by Moishezon's theorem. GAGA algebraizes
the modification and the lifted automorphisms, and Rydh's
Theorem~3.1(ii) gives an equivariant smooth blowup refinement.
This comparison uses Kim's stated conclusion; our proof does not
use his argument. The additional conclusions here are preservation
of $U(X)$, semilinear naturality, and a construction over every
characteristic-zero field. The algebraic construction applies even when
no embedding of the ground field into $\mathbb C$ is available.
For a fixed field $k$, naturality under $k$-isomorphisms alone can be
obtained from equivariant existence by choosing representatives of
isomorphism classes.

\citet[Theorem~1.1 and Remark~1.2]{JiaMeng} construct a K\"ahler
model by smooth blowups equivariant under the stabilizer of a
specified big $(1,1)$-class. Their proof sums translates of an ideal
and uses the unit loci of the residual ideals
\citep[\S3, Claim~3.3]{JiaMeng}. We adapt this method algebraically:
the line bundles of our systems need only agree numerically on
curves in the boundary. The resulting summation lemma is proved
in Section~\ref{sec:systems}.

The construction has two stages. In
Section~\ref{sec:preparation}, we repeatedly principalize the product
of the ideals of the complements of all maximal quasi-projective
opens containing the pulled-back core. The product avoids choosing
one such open.
Termination is proved using an auxiliary ideal with projective blowup,
chosen only for the proof. At the generic points of its codimension-two
support, the sequence becomes a sequence of point blowups on regular
surfaces. Each nonterminal iteration lowers the sum of the local
colengths after divisorial factors are removed.
The remaining boundary images then have codimension at least three
in the original $X$.

When $\dim X\le4$, these images are points or curves. In
Section~\ref{sec:completion}, the exceptional divisors and the
degree-one numerical class of a smooth projective curve give ample
numerical classes on the normalizations of the boundary components.
These provide finite bounds for the classes of line bundles restricted to the boundary.
Within one class, the summation lemma allows us to sum the base ideals
of systems giving projective embeddings of their blowups
(Definition~\ref{def:admissible}), while retaining a projective blowup.
Taking the product
over the finitely many selected classes and opens removes both
choices. This ideal descends to $k$; its principalization gives
the projective endpoint.

\section{Admissible systems and sums of base ideals}\label{sec:systems}

\subsection{The open-preserving Chow lemma}
Let $Y$ be a smooth proper $k$-variety and let $W\subset Y$ be a
dense quasi-projective open. A \emph{linear system} on $Y$ is a
pair $(L,H)$ consisting of a line bundle $L$ and a nonzero finite-dimensional subspace
$H\subset H^0(Y,L)$. Its \emph{base ideal} is
\[
 \cI(L,H)=\operatorname{im}\bigl(H\otimes_k L^{-1}\longrightarrow\cO_Y\bigr).
\]
The associated rational map to $\PP(H^\vee)$, the projective space of
one-dimensional subspaces of the dual of $H$, extends to the blowup
$q:P=\Bl_{\cI(L,H)}Y\to Y$. Its associated line bundle on $P$ is
$q^*L\otimes\cI(L,H)\cO_P$.

\begin{definition}\label{def:admissible}
The linear system $(L,H)$ is \emph{admissible over $W$} if its base ideal is
the unit ideal on $W$ and the induced morphism
$P\to\PP(H^\vee)$ is a closed immersion.
\end{definition}

Rydh calls a blowup $W$-admissible when its center avoids $W$;
our definition also requires the indicated projective embedding.
Thus the ordinary blowup of an admissible base ideal is projective,
and its associated line bundle is very ample. We also use this
definition for smooth proper schemes with finitely many disjoint
integral components, requiring the base ideal to be nonzero on
each component.

\begin{lemma}\label{lem:existence}
An admissible system exists over every dense quasi-projective open
of a smooth proper variety. The same holds for a finite disjoint
union of such varieties, with an open dense in every component.
\end{lemma}
\begin{proof}
The open-preserving Chow lemma gives a projective modification
$\rho:P\to Y$ that is an isomorphism over $W$
\citep[Theorem~3.5]{Rydh}. Take $P$ integral by retaining
the reduced closure of $W$.
Choose a very ample line bundle $A$ on $P$ and set
$L=(\rho_*A)^{\vee\vee}$. The pushforward is torsion-free of
rank one, and regularity of $Y$ makes its reflexive hull an
invertible sheaf. Its global sections contain
$H=H^0(P,A)$. Over $W$ these sections generate $L$.

The graph of the resulting rational map is $P$: the morphism
$(\rho,\iota):P\to Y\times\PP(H^\vee)$ is a closed immersion,
where $\iota$ is the embedding given by $A$, and its image is the
closure of the graph over $W$. The graph of a linear system is
its ordinary base-ideal blowup. This proves admissibility.
For a finite disjoint union, choose a very ample bundle on the
projective disjoint union and use the same argument.
\end{proof}

\begin{lemma}\label{lem:boundary}
Every maximal quasi-projective open of a smooth proper $k$-variety $Y$
has complement of codimension at least two. Consequently $U(Y)$
contains every codimension-one point of $Y$.
\end{lemma}
\begin{proof}
For such an open $W$, take an admissible base ideal $\cI$ from
Lemma~\ref{lem:existence}. Its invertibility locus is isomorphic to
an open of the projective variety $\Bl_{\cI}Y$, so is
quasi-projective. It contains $W$ and every codimension-one point,
since the local ring at such a point is a discrete valuation ring.
Maximality identifies this locus with $W$.
\end{proof}

\subsection{Numerical summation}
Summing base ideals removes the choice of a linear system, provided
projectivity survives the sum. The following lemma supplies the
needed condition: agreement of degrees along the entire boundary.

For a proper integral curve $C$ and a line bundle $L$, write
$L\cdot C=\deg(L|_C)$. In intersection calculations we write line
bundles additively; thus $L-D$ means $L\otimes\cO(-D)$.
We write $c_1(L)$ for the first Chern class of $L$.

\begin{lemma}\label{lem:sum}
Let $F$ be an algebraically closed field of characteristic zero,
$Y$ a smooth proper $F$-variety, and $B\subsetneq Y$ a closed subset.
Let $\{(L_i,H_i)\}_{i\in I}$ be a nonempty family of admissible
systems over $Y\setminus B$, with base ideals $\cI_i$.
Suppose
\[
 L_i\cdot C=L_j\cdot C
 \quad\text{for every integral curve }C\subset B
 \text{ and all }i,j.
\]
Then the coherent ideal $\cJ=\sum_{i\in I}\cI_i$ has projective
ordinary blowup. The assertion also holds componentwise on a
finite disjoint union of smooth proper varieties.
\end{lemma}
\begin{proof}
Noetherianness makes $\cJ$ a finite subsum, so assume $I$ is finite.
Put $\pi:T=\Bl_{\cJ}Y\to Y$ and
$\cJ\cO_T=\cO_T(-D)$. For each $i$ set
\[
 M_i=\pi^*L_i-D,\qquad
 R_i=\cI_i\cO_T\otimes\cO_T(D).
\]
The ideals $R_i$ sum to $\cO_T$. Their unit loci $\Omega_i$ therefore
cover $T$, and all contain $\pi^{-1}(Y\setminus B)$.

Write $q_i:P_i=\Bl_{\cI_i}Y\to Y$ and
$A_i=q_i^*L_i\otimes\cI_i\cO_{P_i}$; $A_i$ is very ample.
Let $Q_i$ be the reduced closure of the common dense open in
$T\times_Y P_i$, with projections
\[
\begin{CD}
 Q_i @>{b_i}>> P_i\\
 @V{a_i}VV @VV{q_i}V\\
 T @>{\pi}>> Y.
\end{CD}
\]
The variety $Q_i$ is projective, since $b_i$ is projective and
$P_i$ is projective. This supplies a projective variety for intersection
calculations before projectivity of $T$ is known.
The morphism $a_i$ is an isomorphism over
$\Omega_i$. Pulling back $\cI_i\subset\cJ$ gives an effective Cartier
divisor $\Delta_i$ on $Q_i$, disjoint from $a_i^{-1}(\Omega_i)$, such that
\begin{equation}\label{eq:comparison}
 a_i^*M_i=b_i^*A_i+\Delta_i.
\end{equation}

We first show that every $M_i$ is nef, meaning that it has
nonnegative degree on every integral curve. If $C\subset T$ is
contracted by $\pi$, then $M_i\cdot C>0$ because $-D$ is
$\pi$-ample. Otherwise choose $i$ with the generic point of $C$
in $\Omega_i$. Its strict transform in $Q_i$ is not contracted by
$b_i$, so \eqref{eq:comparison} gives $M_i\cdot C>0$.
If $\pi(C)\subset B$, the hypothesis and the projection formula
give the same degree for every $M_j$. If $\pi(C)\not\subset B$,
the generic point of $C$ belongs to every $\Omega_j$, and the same
argument applies to each index. Thus every $M_j$ has positive
degree on every curve.

It remains to check higher-dimensional subvarieties. Let
$Z\subset T$ be integral of dimension $d>0$, choose $i$ with
its generic point in $\Omega_i$, and let $Z'\subset Q_i$ be its strict
transform. On $Z'$ write
\[
 M'=a_i^*M_i|_{Z'},\quad N'=b_i^*A_i|_{Z'},\quad
 G=\Delta_i|_{Z'},\quad e=\dim b_i(Z').
\]
Here $M',N'$ are nef and $G$ is effective Cartier, with $M'=N'+G$.
For a point $b\in P_i$ above $y\in Y$, the fiber of $b_i$
is a closed subscheme of
$T_y\times_{\Spec\kappa(y)}\Spec\kappa(b)$, where $\kappa(-)$
denotes the residue field. On that fiber $a_i^*M_i$ is the restriction
of the ample bundle $-D$, up to a trivial factor.
In particular $M'$ is ample on the generic fiber of
$Z'\to b_i(Z')$.

The projection formula now gives
\[
 (M')^{d-e}(N')^e\cdot Z'>0.
\]
Indeed the pushforward of
$c_1(M')^{d-e}\cap[Z']$ is a positive multiple of $[b_i(Z')]$:
its coefficient is the positive top degree on the generic
fiber, or the degree of the generically finite map if $d=e$.
Intersecting with the ample class $c_1(A_i)^e$ gives the
displayed positivity. This also covers $e=0$.
Effectivity of $G$ and nefness on the projective variety $Z'$
therefore imply
\begin{equation}\label{eq:mixed}
 (M')^d\cdot Z'\ \ge\ (M')^{d-1}N'\cdot Z'\ \ge\ \cdots\
 \ge\ (M')^{d-e}(N')^e\cdot Z'\ >0.
\end{equation}
Since $a_i:Z'\to Z$ is birational, this is $M_i^d\cdot Z>0$
\citep[Tag~0BET]{Stacks}.
If $\pi(Z)\subset B$, the boundary-degree hypothesis makes
$a_i^*M_j|_{Z'}$ numerically equivalent to $M'$ for every $j$.
Their top intersections agree on the projective $Z'$, and the
projection formula gives $M_j^d\cdot Z=M_i^d\cdot Z>0$.
Otherwise the generic point of $Z$ belongs to every $\Omega_j$, and each index
can be used from the start.

Nakai--Moishezon on complete schemes makes every $M_j$ ample
\citep{Kleiman}; see also \citet[Theorem~1.3]{FujinoMiyamoto}.
Hence $T$ is projective.
The proof applies separately to the disjoint integral components.
\end{proof}

\section{Finite preparation in every dimension}\label{sec:preparation}

The construction below repeatedly principalizes boundary ideals.
Its termination is measured at a fixed finite set of codimension-two
points of the original $X$, so all boundary images are taken in $X$.

\begin{proposition}\label{prop:preparation}
For every smooth proper $k$-variety $X$, there is a finite natural
sequence of smooth blowups $f:Y\to X$, preserving $U(X)$, such
that either $Y$ is projective or some maximal quasi-projective open
$W\subset Y$ contains $f^{-1}U(X)$ and every component of
$Y\setminus W$ has image of codimension at least three in $X$.
The sequence is empty if $X$ is projective.
\end{proposition}
\begin{proof}
Set $V=U(X)$. At the beginning of a round, let
$f_i:Y_i\to X$ be the sequence already constructed and put
$V_i=f_i^{-1}V$. Let $\cG_i$ consist of all maximal
quasi-projective opens of $Y_i$ containing $V_i$.
It is finite by Benoist's theorem and nonempty since
$V_i\simeq V$ is quasi-projective.
Stop if $Y_i$ is projective or if some $W\in\cG_i$ satisfies
the required image condition. Otherwise principalize
\begin{equation}\label{eq:preparation-ideal}
 \cJ_i=\prod_{W\in\cG_i}\cI_{(Y_i\setminus W)_{\mathrm{red}}},
\end{equation}
where $\cI_Z$ denotes the defining ideal of a closed subscheme
$Z$. The product uses all members of $\cG_i$, so no open is chosen.
Append the whole principalization sequence and begin the next round.
All these ideals are nonzero and unit on $V_i$.

We prove termination using an auxiliary ideal that is not part of
the rule. By Lemma~\ref{lem:existence}, there is a nonzero ideal
$\cI$ unit on $V$ with projective blowup $P=\Bl_{\cI}X$.
By Lemma~\ref{lem:boundary}, its support
$\Supp(\cO_X/\cI)$ has no divisorial component.
Let $S$ range over its finitely many codimension-two components,
and let $\eta_S$ be the generic point of $S$.
The ring $R_S=\cO_{X,\eta_S}$ is regular local of dimension two.

The localized sequence
$T_{i,S}=Y_i\times_X\Spec R_S$ consists, after identity steps
are omitted, of point blowups on a regular surface. The centers
avoid $V_i$, and $V$ contains every codimension-one point of $X$.
On $\Spec R_S$ their images therefore lie over its closed point
$\eta_S$. A center component meeting $T_{i,S}$ must dominate $S$; its
dimension is at most $\dim X-2=\dim S$, so its generic fiber
over $S$ is a finite reduced set of closed points. Divide
$\cI\cO_{T_{i,S}}$ by its full divisorial factor. The residual
ideal has finite colength at finitely many points. Here the
\emph{colength} of an ideal $J$ in a local ring $R$ is the module
length $\length_R(R/J)$, the number of factors in a composition series
when this length is finite. Let $\delta_i$ be the sum of these local
colengths, over all $S$. It is a nonnegative integer.

The local inequality is as follows. Let $J$ be a nonunit ideal
of finite colength in a regular local ring $R$ of dimension two.
On the blowup of its closed point, let $J'_q$ be the residual
ideals after removing the exceptional factor, and let $R'_q$ be
the local rings at their support points. The surface inequality
of \citet[Tag~0AGT]{Stacks} gives
\[
 \sum_q\length_{R'_q}(R'_q/J'_q)<\length_R(R/J).
\]
Applied to the residual ideals above,
this says that blowing up a residual base point strictly decreases
the total colength. Blowing up elsewhere does not increase it,
and previously invertible factors simply pull back.

Suppose a round is nonterminal. The invertibility locus $O$
of $\cI\cO_{Y_i}$ contains $V_i$. It is quasi-projective:
the blowup of $\cI\cO_{Y_i}$ is an isomorphism over $O$
and maps projectively to $P$, by its closed graph in the
fiber product over $X$. Enlarge $O$ to a maximal
quasi-projective $W_{\cI}\in\cG_i$.
Failure of the stopping condition gives a component
$C$ of $Y_i\setminus W_{\cI}$ with codimension-two image $S$.
There can be no codimension-one image by
Lemma~\ref{lem:boundary}. The image is contained in the support
of $\cO_X/\cI$, so this $S$ belongs to the fixed finite list.
Lemma~\ref{lem:boundary} on $Y_i$ gives
$\dim S\le\dim C\le\dim X-2=\dim S$. Thus the generic
point of $C$ becomes a closed point $p$ of $T_{i,S}$ where
$\cI$ is not invertible.

Each reduced complement occurring in
\eqref{eq:preparation-ideal} is either absent or a reduced
zero-dimensional subscheme after localization at $p$. Its
ideal at $p$ is therefore either $\cO_{T_{i,S},p}$ or
the maximal ideal $\mathfrak m_p$. The factor from
$W_{\cI}$ is the latter. Thus
$(\cJ_i)_p=\mathfrak m_p^a$ for some $a>0$.
Since $\mathfrak m_p^a$ is not invertible, principalization must
eventually have a nonidentity step whose center meets the fiber over $p$. Before the first such step, all centers are
disjoint from that fiber, so the localized surface and residual ideal
are unchanged near $p$. That first step therefore blows up $p$ while
the residual ideal of $\cI$ is still nonunit, and strictly lowers the
total colength. The other steps cannot increase the total colength, so
$\delta_{i+1}<\delta_i$.
A nonnegative integer cannot decrease through infinitely many
rounds, proving termination.

The ideals and stopping conditions are invariant under
isomorphisms of the input, so the fixed principalization
procedure makes the entire sequence natural. Its centers
avoid the inverse image of $V$ throughout.
\end{proof}

\section{Numerical selection and descent}\label{sec:completion}

Lemma~\ref{lem:sum} removes the choice of a system within one
boundary numerical class. We must now specify finitely many such
classes without choosing a polarization. Assume that the prepared
boundary images have dimension at most one. The prepared sequence
provides a relatively ample bundle; the point-or-curve image
condition gives ample numerical classes on the boundary normalizations.

\subsection{An ample class from the sequence}
Let $f:Y\to X$ be the prepared sequence and
$V=U(X)$. If $Y$ is projective, no further step is needed.
Write $E_1,\ldots,E_s$ for the total transforms in $Y$ of the
exceptional divisors of its successive nonidentity blowups, in
that order. Some positive integer tuple $\mathbf a=(a_1,\ldots,a_s)$
makes
\[
 A=\cO_Y\Bigl(-\sum_{i=1}^s a_iE_i\Bigr)
\]
$f$-ample, by composition of relative ampleness
\citep[Tag~0C4K]{Stacks}. Use the lexicographically least such tuple:
it exists by successively minimizing the finitely many coordinates.
An isomorphism preserves the ordered divisors and relative ampleness,
so it preserves this tuple and $A$. For an empty sequence take
$A=\cO_Y$, which is ample relative to the identity.

We pass to an algebraic closure to use numerical divisor groups
and the degree-one class on a curve; over $k$, a curve need not
have a divisor of degree one. Let $F$ be an algebraic closure of
$k$ and write
$Y_F=Y\times_k F$, $X_F=X\times_k F$, and $f_F:Y_F\to X_F$.
The scheme $Y_F$ is a finite disjoint union of smooth integral
varieties. Put $V_F=f_F^{-1}(V\times_k F)$, and let $A_F$ be the
pullback of $A$ to $Y_F$.
Let $\cG$ consist of all maximal quasi-projective opens
$W\subset Y_F$ containing $V_F$ for which every boundary
component has image of dimension at most one in $X_F$.
If the prepared boundary images over $k$ have dimension at most
one, then $\cG$ is nonempty: extend the base change of an
open supplied by Proposition~\ref{prop:preparation} to a maximal
open. Its complement only shrinks.
The set $\cG$ is finite by Benoist's theorem.

Fix $W\in\cG$ and put $B=(Y_F\setminus W)_{\mathrm{red}}$.
The boundary classes and bounds below are constructed for this
$W$; their dependence on $W$ is suppressed until the final product.
For each positive-dimensional irreducible component $B_j$ of
$B$, let $\nu_j:Z_j\to B_j$ be its normalization and let $D_j$
be its reduced image in $X_F$. Both $B_j$ and $Z_j$ are
projective: $B_j\to D_j$ is projective, normalization is finite,
and a proper integral curve or point over $F$ is projective.
We use $A_F$ also for its pullbacks to these normalizations.

For a projective variety $Z$, let $N^1(Z)_{\ZZ}$ denote the
group of line bundles modulo equality of degrees on all
integral curves, and set
$N^1(Z)_{\RR}=N^1(Z)_{\ZZ}\otimes_{\ZZ}\RR$.
The integral group is a finite-rank lattice: numerical equivalence
is coarser than algebraic equivalence, and the numerical quotient
is torsion-free, so the theorem of the base applies
\citep[Theorem~7.4]{ChengEtAl}.
Ampleness depends only on numerical class.

If $D_j$ is a point, set $h_j=[A_F|_{Z_j}]$; this is ample
since $B_j$ lies in a fiber of $f_F$ and $A_F$ is $f_F$-ample.
If $D_j$ is a curve, normality gives a morphism
$g_j:Z_j\to D_j^\nu$ to its normalization. The smooth
projective curve $D_j^\nu$ has a distinguished numerical
class $\kappa_j$ of degree one. Put
\begin{equation}\label{eq:boundary-class}
 h_j=[A_F|_{Z_j}]+n_jg_j^*\kappa_j,
\end{equation}
where $n_j$ is the least nonnegative integer making this class
ample. Such an integer exists: the restriction of $A_F$ is
$g_j$-ample, and relative ample twisting applies
\citep[Tags~035Q, 0C4L and 0892]{Stacks}.
Only the numerical class of degree one enters the definition.
All the classes $h_j$ are ample and are preserved, with
their indices, by isomorphisms of the data.

\subsection{Finite numerical bounds}
Define an equivalence relation on $\Pic(Y_F)$ by
$L\sim_B L'$ if $L\cdot C=L'\cdot C$ for every integral
curve $C\subset B$, and let
\[
 \Gamma_B=\Pic(Y_F)/{\sim_B}.
\]
Write $[L]_B$ for the class of $L$ in this quotient.
Restriction induces an injection
\begin{equation}\label{eq:lattice}
 \Gamma_B\ \lhook\joinrel\longrightarrow\
          \bigoplus_j N^1(Z_j)_{\ZZ}.
\end{equation}
Indeed a curve in $B$ lifts to a curve of positive finite
degree on one of its normalizations, and a curve in a
normalization maps finitely to a boundary curve. The degree
formula identifies the two kernels. Thus $\Gamma_B$ is a
finite-rank lattice in the real span of its image.
Zero-dimensional boundary components impose no conditions.

For each integer $N\ge0$ define
\begin{equation}\label{eq:bounds}
 K_B(N)=\{\gamma\in\Gamma_B:
 Nh_j+\gamma|_{Z_j},\ Nh_j-\gamma|_{Z_j}
            \text{ are nef for every }j\}.
\end{equation}
Here nef means nonnegative on all integral curves of $Z_j$;
write $\Nef(Z_j)$ for the cone of these classes in $N^1(Z_j)_{\RR}$.

\begin{lemma}\label{lem:bounds}
The sets $K_B(N)$ are finite, nested, and exhaust $\Gamma_B$.
Their definition is invariant under isomorphisms, including
semilinear isomorphisms of the geometric data.
\end{lemma}
\begin{proof}
In $N^1(Z_j)_{\RR}$ the nef cone is closed and pointed,
and $h_j$ lies in its interior. The closed set of classes
$u$ satisfying $Nh_j\pm u\in\Nef(Z_j)$ is bounded.
Otherwise a normalized unbounded sequence would have a
nonzero limit for which both $u$ and $-u$ are nef,
contradicting pointedness. It is therefore compact.
The product of these compact sets meets the lattice
\eqref{eq:lattice} in finitely many points.
For any fixed class $\gamma$, openness of the ample cone
gives $Nh_j\pm\gamma|_{Z_j}$ ample for all sufficiently
large $N$. This proves exhaustion. Nestedness follows
by adding the nef class $h_j$. All defining conditions
are intrinsic to the indicated data.
\end{proof}

\subsection{The ideal and its descent}
For each $W\in\cG$, again put $B=(Y_F\setminus W)_{\mathrm{red}}$.
Let $N_W$ be the least $N$ for which
$K_B(N)$ contains the class of a line bundle occurring in
an admissible system over $W$. It exists by
Lemma~\ref{lem:existence} and Lemma~\ref{lem:bounds}.
Let $S_W$ be the nonempty finite set of all classes in
$K_B(N_W)$ that occur. For $\gamma\in S_W$, define
\[
 \cJ_{W,\gamma}=
 \sum_{\substack{(L,H)\text{ admissible over }W\\
                 [L]_B=\gamma}}\cI(L,H).
\]
Noetherianness makes this a coherent ideal and a finite
subsum. Lemma~\ref{lem:sum} shows that its ordinary blowup
is projective. Define the finite product
\begin{equation}\label{eq:final-ideal}
 \cK_F=\prod_{W\in\cG}\ \prod_{\gamma\in S_W}\cJ_{W,\gamma}.
\end{equation}
It is unit on $V_F$ and nonzero on each component of $Y_F$.

\begin{lemma}\label{lem:descent}
The ideal $\cK_F$ descends uniquely to a coherent ideal
$\cK\subset\cO_Y$. Principalizing $\cK$ over $k$ gives a
smooth variety projective over $k$, and all its centers
avoid $f^{-1}V$.
\end{lemma}
\begin{proof}
Each $\sigma\in\Gal(F/k)$ preserves $V_F$ and $A_F$ and
permutes $\cG$. It transports boundary normalizations, the classes
$h_j$, the bounds $K_B(N)$ and admissible systems to the corresponding
data for $\sigma(W)$. Hence $N_{\sigma(W)}=N_W$ and
$S_{\sigma(W)}=\sigma(S_W)$, and it carries $\cJ_{W,\gamma}$
to $\cJ_{\sigma(W),\sigma(\gamma)}$.
Thus $\sigma$ permutes the factors of \eqref{eq:final-ideal}:
$\cK_F$ is an invariant ideal, not merely an invariant
isomorphism class of sheaves.

Descend its finite presentation to a
finite Galois extension $E/k$ inside $F$, obtaining
$\cK_E$. Invariance descends by faithful flatness.
On an affine open $\Spec R\subset Y$, choose a $k$-basis
$e_1,\ldots,e_t$ of $E$ and its trace-dual basis
$d_1,\ldots,d_t$. For any element
$u=\sum_\ell r_\ell\otimes e_\ell$ of the ideal $\cK_E$,
\[
 \sum_{\tau\in\Gal(E/k)}\tau(d_j u)=r_j\otimes1.
\]
The right side therefore belongs to $\cK_E$.
It follows that $\cK_E$ is the extension of its contraction
to $R$. These contractions are coherent and glue.
Faithful flatness also proves uniqueness, and hence
independence of the algebraic closure used in the construction.

Principalize $\cK$ over $k$, with endpoint $Y'\to Y$, and put
$Y'_F=Y'\times_k F$.
After base change to $F$, the pullback of
\eqref{eq:final-ideal} to $Y'_F$ is invertible. In a local domain,
if a product of nonzero ideals is invertible, each factor is invertible: after
dividing the product of the other factors by a generator
of the product, one obtains an inverse fractional ideal.
Since $Y'_F$ is smooth, this applies on every component.
Thus $Y'_F$ maps to each
$\Bl_{\cJ_{W,\gamma}}Y_F$.
The map is projective, by its closed graph in the fiber
product over $Y_F$, because $Y'_F\to Y_F$ is projective.
Its target is projective, so $Y'_F$ is projective.
Existence of an ample line bundle descends from $F$ to $k$
\citep[Tags~0BDC and 0BDE]{Stacks}; properness of $Y'$
then gives projectivity over $k$.
Finally, $\cK$ is unit on $f^{-1}V$, so every
principalization center avoids that open.
\end{proof}

\begin{proof}[Proof of Theorem~\ref{thm:main}]
Run Proposition~\ref{prop:preparation}. If the endpoint is
projective, stop. Otherwise the relevant boundary images
have codimension at least three in $X$. Since $\dim X\le4$,
they have dimension at most one, so the construction of
Section~\ref{sec:completion} applies. Lemma~\ref{lem:descent}
gives the projective endpoint over $k$.

The preparation, the numerical selection and the final
principalization are finite. All centers avoid the inverse image of the original
$U(X)$. Since the ideals are generically nonzero, no center
contains an ambient component. After omitting the identity blowups
as agreed above, all centers have codimension at least two.

An isomorphism of inputs transports the preparation, its
ordered exceptional divisors, the finite numerical sets,
and the actual descended ideal. Functorial principalization
then transports every stage. The lifts respect identities and
composition by uniqueness on
the dense unchanged open. For an isomorphism over $k\simeq k'$,
extend the field isomorphism to algebraic closures; invariance and
uniqueness of descent in Lemma~\ref{lem:descent} make the descended
ideal independent of this extension.
The initial stopping condition makes the sequence empty
on projective inputs.
\end{proof}

\begin{remark}
The passage to $F$ is used only to select an ideal, which is
then descended and principalized over $k$. The construction
keeps the pullback of $U(X)$ throughout; it does not
replace it by the core computed over $F$.
In dimension at least five, preparation still terminates,
but boundary images may be surfaces. The degree-one
numerical class used in \eqref{eq:boundary-class} then has
no corresponding construction here. The present argument
does not give projectivization in that dimension range.
\end{remark}

\section*{Acknowledgements}
The author received no funding for this work and declares no conflicts
of interest.

\section*{AI Declaration}
AI tools were used extensively in the development and preparation of this
work. OpenAI Astra (\texttt{gpt-6-astra}, accessed through Codex) was used to
explore proof strategies, develop and revise mathematical arguments,
search and compare the literature, draft the manuscript and its \LaTeX{}
source, and check the arguments and exposition. Anthropic Claude
(Opus~5.5) was used to discuss candidate constructions and to assess
proof arguments, literature comparisons and successive manuscript drafts.
Codex also assisted with incorporating the author's revisions and preparing the submission files. Claude's Haiku~4.5 was also used internally by web tools during journal
selection. The author directed the work and takes responsibility for the
mathematical content, citations and final presentation.

\section*{Data Access Statement}
The results are established by the proofs in this article and the cited
literature. No external datasets or computational certificates are required
to verify them.

\end{document}